\documentclass{amsart}

\usepackage{graphicx}

\usepackage{subfig}
\usepackage{amsmath}
\usepackage{mathrsfs} 
\usepackage{amsfonts}
\usepackage{amssymb}%
\usepackage{amsthm}
\usepackage{amscd}
\usepackage{verbatim}
 \usepackage{mathabx}
 \usepackage{dsfont}
\usepackage{endnotes}
\usepackage{enumerate}
\usepackage{mathtools}
\usepackage{comment} 
\usepackage{tikz}
\usetikzlibrary{shapes,arrows,calc,positioning}
\usepackage{color}
\usepackage{hyperref}
\hypersetup{
  pdfauthor={Jacob Turner},
  pdftitle={Tensors Masquerading as Matchgates},
  pdfsubject={},
  urlcolor=blue,
}

\newcommand{\mcP}{\mathcal{P}}

\newcommand{\C}{\mathbb{C}}

\newcommand{\N}{\mathbb{N}}

\renewcommand{\P}{{\sf{P}}}
\newcommand{\NP}{{\sf{NP}}}

\newcommand{\ot}{\otimes}

\newcommand{\GL}{\operatorname{GL}}

\newcommand{\SL}{\operatorname{SL}}

\newcommand{\sPf}{\operatorname{sPf}}
\newcommand{\sPfd}{\operatorname{sPf}^{\vee}}
\newcommand{\Pf}{\operatorname{Pf}}

\newcommand{\tn}{\textnormal}

\newcommand{\End}{\textnormal{End}}
\newcommand{\ott}{\bigotimes}

\newcommand{\op}{\oplus}
\newcommand{\bop}{\bigoplus}

\newcommand{\actson}{\curvearrowright}

\theoremstyle{plain}
\newtheorem{theorem}{Theorem}[section]
\newtheorem{corollary}[theorem]{Corollary}
\newtheorem{proposition}[theorem]{Proposition}
\newtheorem{lemma}[theorem]{Lemma}

\newtheorem{problem}[theorem]{Problem}
\theoremstyle{definition}
\newtheorem{defn}[theorem]{Definition}
\theoremstyle{definition}
\newtheorem{observation}[theorem]{Observation}
\theoremstyle{definition}

\theoremstyle{definition}
\newtheorem{example}[theorem]{Example}
\theoremstyle{definition}

\theoremstyle{definition}

\title[Tensors Masquerading as Matchgates]{Tensors Masquerading as Matchgates: Relaxing Planarity Restrictions on Pfaffian Circuits}
\author[Jacob Turner]{Jacob Turner${}^*$}
\thanks{${}^*$ Korteweg-de Vries Institute for Mathematics, University of Amsterdam, 1098 XG Amsterdam, Netherlands.}
\begin{document}
 \begin{abstract}
  Holographic algorithms, alternatively known as Pfaffian circuits, have received a great deal of attention for giving polynomial-time algorithms of $\#\P$-hard problems. Much work has been done to determine the extent of what this machinery can do and the expressiveness of these circuits. One aspect of interest is the fact that these circuits must be planar. Work has been done to try and relax the planarity conditions and extend these algorithms further. We show that an approach based on orbit closures does not work, but give a different technique for allowing the SWAP gate to be used in a Pfaffian circuit given a suitable basis and restricted type of graph. This is done by exploiting the fact that the set of  Pfaffian (co)gates always lies in a hyperplane. We then give a variety of bases that can be chosen such that the SWAP gate acts like a Pfaffian cogate and discuss how many SWAP gates can be implemented in a Pfaffian circuit.
  
 \end{abstract}
  \maketitle


\section{Introduction}

Leslie Valiant defined a set of linear operators, which he called matchgates, as a set of building blocks from which could be built circuits (which he called matchcircuits) \cite{QCtcbSiPT}. One of the main motivations for matchcircuits is that they are akin to how quantum circuits are defined, but can be computed efficiently \cite{JozsaMiyake}. This relationship is especially clear when matchcircuits are place within the formalism of tensor networks (we shall assume the reader is familiar with these objects) \cite{landsberg2012holographic,morton2010pfaffian}. 

Matchgates are defined by a set of equations, which depends on the number of inputs and outputs of the gate. That is, for a fixed number of inputs/outputs, the gates define a variety \cite{MR1932906,cai2007theory}. The structure of these gates is such that matchcircuits can be computed in polynomial time. This stands in contrast to the complexity of evaluating an arbitrary tensor network, a problem known to be $\#\P$-hard \cite{damm03}.

There is a natural group action on any matchcircuit that leaves the value of the computation unchanged. As such, the orbit of any circuit is another circuit with a polynomial-time evaluation. These are called holographic algorithms (in the setting of tensor networks, they go by the name of Pfaffian circuits and is the name we shall use) and were used by Valiant to devise polynomial time algorithms for problems not known to be in $\P$ and which are closely related to $\NP$ and $\#\P$-hard problems \cite{ValiantFOCS2004}. 

 Pfaffian circuits have the restriction that they are seemingly planar. The map that allows two wires to be switched is called the SWAP gate (or tensor) and acts on a basis of $\C^2\ot \C^2$ by $e_i\ot e_j\mapsto e_j\ot e_i$. It is known that this tensor is not a Pfaffian gate in any basis \cite{margulies2013polynomial}. However, it may the case that a combination Pfaffian gates in some basis can be composed to form the SWAP gate; this question is still open.
 
 Holographic algorithms have also been looked at in the broader context of $\P$ v.s. $\#\P$ in complexity theory, specifically with regard to dichotomy theorems for counting constraint satisfaction problems (\#CSP). 
 
 Given a \#CSP problem, one can visualize it with a bipartite graph: one independent set corresponds to the variables and the other independent set corresponds to the clauses. A variable vertex is connected to a clause vertex if the variable appears in the clause. If there are no restrictions on the graph, those clauses (which can be viewed as tensors or gates) for which the corresponding \#CSP problem can be solved efficiently have been classically. Otherwise, the problem is as hard as any in $\#\P$, assuming $\P\ne\#\P$ \cite{bulatov2010complexity,dyer2010effective,cai2011non,cai2012complexity}.
 
Not all matchgates are  among the allowed gates specified by the aforementioned dichotomy theorem. Despite the fact that \#CSP problems built from these gates are efficiently solvable, there is no contradiction as the graphs associated to the problem are forced to be planar. In fact, the tractable planar-\#CSP problems are precisely those holographic algorithms built from matchgates \cite{cai2010holographic}. This pushes the boundary of polynomial-time solvable \#CSP problems. This could potentially be pushed even further if planarity restrictions were somehow relaxed.
 
 In this paper, we discuss two strategies for allowing a Pfaffian circuit to have SWAP tensors while retaining the tractability of the evaluation problem. This will be done by looking at Pfaffian gates that act identically to the SWAP gate in the context of a given circuit. 
 
 In fact, these two strategies can be used in trying to ``algebraically approximate" a tensor network with a Pfaffian circuit. By this, we mean that there is an algebraic construction that changes the gates in a Pfaffian circuit to tensors that are not normally allowed, but because of the nature of the algebraic construction, preserves the value of the Pfaffian circuit it came from. We hope  these techniques can be used to replace Pfaffian gates with more familiar or convenient gates, aiding in the combinatorial reasoning about such circuits and in designing polynomial time algorithms.
 
 The first strategy involves looking at orbit closures of Pfaffian gates. As already mentioned, there is a natural group action on Pfaffian circuits that does not change its value. Furthermore, if one replaces the gates with those in the Zariski closure of an orbit, one gets a circuit with the same value as the original Pfaffian circuit. Thus one may look to see if the SWAP gate lies in the orbit closure of any Pfaffian gate. If so, then this gate could be replaced with SWAP gate without affecting the value of the circuit, perhaps assuming some special conditions.
 
 The second strategy uses the fact that if a basis is specified, the variety of Pfaffian gates with fixed inputs/outputs lie in a hyperspace. The last step in evaluating a Pfaffian circuit is done by taking the inner product of two Pfaffian gates. But since the set of Pfaffian gates lie in a hyperspace, it may be that $\langle v-u,x\rangle=0$ for all Pfaffian gates $x$, and a particular Pfaffian gate $u$ and tensor $v$. Then one can replace the Pfaffian gate $u$ with the tensor $v$ and leave the value of the circuit unchanged. 
 
 The organization of this paper is as follows: We first give the necessary background on Pfaffian circuits,  the group action on these circuits, as well as the evaluation algorithm and a discussion of the relevant varieties. In Section \ref{sec3}, we discuss the problem of determining which Pfaffian gates contain the SWAP gate in their orbit closures. We conclude that no such Pfaffian gate exists. Lastly, in Section \ref{sec4}, we determine a set of basis changes for certain Pfaffian circuits such that certain Pfaffian gates can be replaced by the SWAP gate. We show however, that at most one SWAP gate can be replaced. 
 
 \section{Background}\label{sec:background}
  
 A Pfaffian circuit is given as a planar bipartite graph with a tensor on each vertex. The variety of tensors that can be placed on a vertex depends on which independent set the vertex is in, as well as its degree. Whereas in the introduction, we referred to all of these tensors as Pfaffian gates, henceforth the tensors on one independent set will be called \emph{Pfaffian gates} and the tensors on the other independent set \emph{Pfaffian cogates}. 
 
  Throughout this paper, we use bra-ket notation. Let $\{v_0,v_1\}$ be a orthonormal basis of $\C^2$ and $\{v_0^*,v_1^*\}$ be the respective dual basis. By $|i_1\cdots i_k\rangle$, $i_j\in\{0,1\}$, we mean $\ot_{j=1}^k{v_{i_j}}$ and $\langle i_1\cdots i_k|:=\ot_{j=1}^k{v^*_{i_j}}$.  For $I\subseteq[n]$, let $\chi_I(i)=1$ if $i\in I$ and 0 otherwise, for $i\in[n]$. Then define $|I\rangle:=|\chi_I(1)\cdots\chi_I(n)\rangle$ and $\langle I|$ likewise.
  
  Let $\mathfrak{J}_n$ be the set of $n\times n$ skew-symmetric matrices, and $\mathfrak{J}:=\bigcup_{n=1}^{\infty}{\mathfrak{J}_n}$. Given an $n\times n$ matrix $M$ and $I\subseteq [n]$, let $M_I$ be the principal minor formed by taking the rows and columns in $I$. Let $\overline{I}$ denote the complement of $I$ in $[n]$. Then we define two maps $\sPf_n:\mathfrak{J}_n\to \C^{2^n}$ and $\sPfd_n:\mathfrak{J}\to(\C^*)^{2^n}$ as follows:
 $$\sPf_n(M)=\sum_{I\subseteq [n]}{\Pf(M_I)|I\rangle}$$
 $$\sPfd_n(M)=\sum_{I\subseteq [n]}{\Pf(M_{I})\langle\overline{I}|}$$
 
 By definition, $\Pf(\emptyset)=1$. The images of $\sPf_n(M)$ and $\sPfd_n(M)$ define varieties $\mcP_n$ and $\mcP^\vee_n$ of \emph{(elementary) arity n} gates and \emph{(elementary) arity n} cogates respectively. We will extend our notion of Pfaffian gates and cogates later. These maps lift to maps on all of $\mathfrak{J}$ and the so the set of gates is $\mathcal{P}:=\bigcup{\mcP_n}$ and the set of cogates is $\mcP^\vee:=\bigcup{\mcP^\vee_n}$.
 
 \begin{defn}\label{def:pfaffcircuit}
  A \emph{Pfaffian circuit} is a planar bipartite graph with two independent sets $W_1,W_2$ such that for every $v\in W_1$, $v$ is assigned a tensor in $\mcP_{\tn{deg}(v)}$ and for every $w\in W_2$, $w$ is assigned a tensor in $\mcP^\vee_{\tn{deg}(v)}$.
 \end{defn}

 To evaluate a Pfaffian circuit, it is first embedded in the plane. Since the graph is bipartite, the dual graph is Eulerian. The Eulerian cycle of the dual graph can be drawn as a planar curve that intersects every edge of the Pfaffian circuit exactly once. An edge is then given the label 1. The next edge the planar curve intersects is given the label 2, the next edge label 3, and so on.
 
 Rearranging the graph so that the edges are all placed parallel to each other, aligned vertically such that the labels increase from top to bottom makes one independent set of vertices all lie to the left of some vertical line and the other independent set lie on the other side. Then one of the two independent sets is equal to a single elementary Pfaffian gate and the other a single elementary Pfaffian cogate.
 
 \begin{defn}\label{def:labeldirectsum}
  Let $M$ be a square matrix whose rows and columns are labeled by $I\subset\N$, with the labels in strictly increasing order. Similarly, let $N$ be another such matrix with labels $J$ and $I\cap J=\emptyset$. We define $M\tilde{\op}N$ to be the direct sum $M\op N$ (which inherits a natural labeling from the labeling of $M$ and $N$) rearranged such that rows and columns are labeled in strictly increasing order.
 \end{defn}
 \begin{example}
  \begin{equation*}
 \bordermatrix{ & 1 & 3\cr
1 &m_{11}&m_{12}\cr
 3&m_{21}&m_{22}}\tilde{\op}
 \bordermatrix{ &2&4&5\cr
 2&n_{11}&n_{12}&n_{13}\cr
 4&n_{21}&n_{22}&n_{23}\cr
 5&n_{31}&n_{32}&n_{33}}=
 \bordermatrix{ &1&2&3&4&5\cr
 1&m_{11}&0&m_{12}&0&0\cr
 2&0&n_{11}&0&n_{12}&n_{13}\cr
 3&m_{21}&0&m_{22}&0&0\cr
 4&0&n_{21}&0&n_{22}&n_{23}\cr
 5&0&n_{31}&0&n_{32}&n_{33}}
  \end{equation*}

 \end{example}

 \begin{lemma}[\cite{morton2010pfaffian}]\label{lem:directsum}
  For two skew-symmetric matrices $M,N\in\mathfrak{J}$ labeled as in Definition \ref{def:labeldirectsum}, $\sPf(M)\ot\sPf(N)=\sPf(M\tilde{\op}N)$ and $\sPfd(M)\ot\sPfd(N)=\sPfd(M\tilde{\op}N)$.
 \end{lemma}

 With the labeling the coming from the planar curve, we see that we can combine the two independent sets into a Pfaffian gate $\sPf(\Xi)$ and a Pfaffian cogate $\sPfd(\Theta)$. 
 
 \begin{defn}
    The value of a Pfaffian circuit is given by the value of the standard pairing $$\langle\sPfd(\Theta),\sPf(\Xi)\rangle.$$
 \end{defn}

 \begin{theorem}[\cite{MR1069389,morton2010pfaffian}]\label{thm:pfaffkernel}
$\langle\sPfd(\Theta),\sPf(\Xi)\rangle=\Pf(\Xi+\tilde{\Theta})$ where $\tilde{\Theta}_{ij}=(-1)^{i+j+1}\Theta_{ij}$.
 \end{theorem}

 Theorem \ref{thm:pfaffkernel} tells us that instead of taking the inner product of two vectors of length $2^n$, we can instead compute the Pfaffian of an $n\times n$ skew-symmetric matrix. This allows us to efficiently compute the value of a Pfaffian circuit.
 
 \subsection{The Group Action on Pfaffian Circuits}
 
 Let $\Gamma$ be a Pfaffian circuit and $G=(V,E)$ its underlying graph. Let $V=W_1\cup W_2$ be divided into its two independent sets. Then consider the following vector space: $$U_\Gamma:=\bigg(\bop_{v\in W_1}(\C^2)^{\ot\tn{deg}(v)}\bigg)\op\bigg(\bop_{w\in W_2}{((\C^2)^*)^{\ot\tn{deg}(w)}}\bigg)$$ $$=\bop_{e\in E}(\C^2\op (\C^2)^*).$$ Consider the action of $\SL(2,\C)$ on $\C^2\op(\C^2)^*$ by $g.(v,\phi):=(gv,\phi\circ g^{-1})$. This induces an action on $U_G$ by the group
 $$\SL_\Gamma:=\bop_{e\in E}{\SL(2,\C)}.$$ One could also use the group $\bop_{e\in E}{\GL(2,\C)}$ but this does not add to the number of circuits equivalent to Pfaffian circuits as conjugation by $\GL(V)$ is not a faithful action. 
 
  The Pfaffian (co)gates of $\Gamma$ are all elements of $U_\Gamma$ and the value of $\Gamma$ is invariant under the action of $\GL_\Gamma$ on $U_\Gamma$. So it makes sense to extend our definitions of Pfaffian (co)gates beyond the elementary ones defined previously.
  
  \begin{defn}
   A tensor in $(\C^2)^{\ot n}$ is a Pfaffian gate if it is in the orbit of the group $\SL(2,\C)^{\op n}$ acting on $\mcP_n$. We define a tensor in $((\C^2)^*)^{\ot n}$ to be a Pfaffian cogate if it is in the orbit of $\SL(2,\C)^{\op n}$ acting on $\mcP^\vee_n$. 
  \end{defn}
  
  We do not allow the arbitrary Pfaffian gates and cogates to be associated to vertices in a bipartite graph. There is a compatability condition for $\Gamma$, namely that under the action of $\SL_\Gamma$, $\Gamma$ can be taken to a Pfaffian circuit as in Definition \ref{def:pfaffcircuit}.

\subsection{Equations Defining Pfaffian (Co)Gates}

As previously mentioned, for every $n$, $\mcP_n$ and $\mcP^\vee_n$ form varieties. We now briefly discuss the equations defining these varieties, as done in \cite{landsberg2012holographic}, based off of Valiants original identities \cite{QCtcbSiPT}. Since the SWAP gate is in either $(\C^2)^{\ot 4}$ or $((\C^2)^*)^{\ot 4}$, we will thoroughly examine these particular cases.

Let $P=\sum_{I\subseteq[n]}{\alpha_I|I\rangle}$ be a vector in $(\C^2)^{\ot n}$. The most basic equations that $P$ must satisfy to be in $\mcP_n$ are that $\alpha_I=0$ for $|I|$ odd. The following observation will be important later on.

\begin{observation}\label{obs}
 The variety $\mcP_n$ is contained in a proper linear subspace of $(\C^2)^{\ot n}$.
\end{observation}

Similar linear equations hold for $\mcP^\vee_n$. For $n=0,1,2,3$, these linear equations define all of $\mcP_n$, and their counterparts define all of $\mcP_n^\vee$. For $n\ge 4$, more non-trivial equations appear. For $n=4$, we have the extra equation
$$\alpha_{\emptyset}\alpha_{[4]}=\alpha_{[4]}=\alpha_{\{1,2\}}\alpha_{\{3,4\}}-\alpha_{\{1,3\}}\alpha_{\{2,4\}}+\alpha_{\{2,3\}}\alpha_{\{1,4\}},$$ noting that $\alpha_{\emptyset}=\Pf(\emptyset)=1$. In general, the equations arise from the the parameterizations of $\mcP_n$ by Pfaffians of matrices. The general variety is defined by the Grassmann-Pl\"ucker relations. The variety $\mcP^\vee_4$ is defined by the similar equation
$$\alpha_{\emptyset}\alpha_{[4]}=\alpha_{\emptyset}=\alpha_{\{1,2\}}\alpha_{\{3,4\}}-\alpha_{\{1,3\}}\alpha_{\{2,4\}}+\alpha_{\{2,3\}}\alpha_{\{1,4\}}$$ and the linear equations $\alpha_I=0$ for $|I|$ odd. The Grassmann-Pl\"ucker relations can be slightly adapted to describe the variety $\mcP_n^\vee$ in general. The following equations can be found in \cite{manivel2009spinor}, for example.

\begin{theorem}\label{thm:rels}
 Let $G=\sum_{I\subseteq[n]}{\alpha_I|I\rangle}\in(\C^2)^{\ot n}$. Then it lies in $\mcP_n$ if for every disjoint set of integers $R,S,T$, with $|T|=|R|+4k$ ($k>0$), 
 $$\sum_{\stackrel{A\cup B=R\cup T}{A\cap B=\emptyset}}{\epsilon(S\cup A,S\cup B,R,T)\alpha_{S\cup A}\alpha_{S\cup B}}=0$$
 where $\epsilon(S\cup A,S\cup B,R,T)=\pm 1$ and $\alpha_{\emptyset}=1$. Omitting the condition $\alpha_{\emptyset}=1$ defines the cone over the variety, $\C\mcP_n$. These same equations also define the variety $\mcP_n^\vee$ where $G=\sum_{I\subseteq[n]}{\alpha_i\langle I|}$  except with the above sets replaced with their complements in $[n]$.
\end{theorem}

\section{Orbit Closures of Pfaffian (Co)Gates}\label{sec3}

Theorem \ref{thm:pfaffkernel} implies that the value of a Pfaffian circuit $\Gamma$ is a polynomial in the entries of the Pfaffian (co)gates appearing in the circuit. The set of gates in $\Gamma$ have a group $\SL_\Gamma$ acting on them as previously described in Section \ref{sec:background}. For every set of gates in this orbit, the circuit has the same value. 

Furthermore, since we have a polynomial that takes a constant value on a set, it takes that same value on any point in the Zariski closure of that set. As such, we can extend our definition of Pfaffian circuits even further to include those circuits $\Gamma$ whose gates are in the $\GL_\Gamma$ orbit closure of a set of Pfaffian (co)gates, rather than just the orbit.

The reason for considering this in the context of the SWAP gate is as follows: if there was a Pfaffian circuit as in Definition \ref{def:pfaffcircuit} which included a SWAP gate, then this would imply that the SWAP gate was in the orbit of $\SL(2,\C)^4$ acting on $\mcP_4$ or $\mcP_4^\vee$. However, it has been shown that this is not the case \cite{margulies2013polynomial}.

Under our extended definition, if there is a circuit $\Gamma$ in the orbit closure of a Pfaffian circuit which contains the SWAP gate, this implies that the SWAP gate lies in the orbit closure of $\SL(2,\C)^4$ acting on $\mcP_4$ or $\mcP_4^\vee$. Something similar was done in the context of knot theory to find R-matrices that did not satisfy the Reidemeister moves, yet nonetheless could be used to define polynomial knot invariants \cite{schrijver2012virtual}.

In this setting, we shall see that the SWAP gate does not lie in the orbit closure of any Pfaffian (co)gate. To show this, we will show that no Pfaffian (co)gate agrees with the SWAP gate on the polynomial invariants of $\SL(2,\C)^4\actson (\C^2)^{\ot 4}$ or $\SL(2,\C)^4\actson((\C^2)^*)^{\ot 4}$. The polynomial ring of the first action is known to be generated by four algebraically independent invariants. Using the self-duality of $\SL(2,\C)$, we can then deduce the four algebraically independent invariants that generate the invariant ring of the latter action.

\subsection{The polynomial invariants of $\SL(2,\C)^4$ acting on $(\C^2)^{\ot 4}$ and its dual}

The generators for $\SL(2,\C)^4\actson(\C^2)^{\ot 4}$ were described in \cite{MR2039690}. We first describe these invariants and then determine what values they take on the SWAP gate. We identify the coordinate ring of $(\C^2)^{\ot 4}$ with the ring $\C[x_1,\dots,x_{16}]$. The first invariant is one of Cayley's hyperdeterminants:
\begin{center}
 \begin{tabular}{l l}
  $H(x_1,\dots,x_{16})=$ & $\quad x_1x_{16}-x_2x_{15}-x_3x_{14}+x_4x_{13}$\\
   & $\;-x_5x_{12}+x_6x_{11}+x_7x_{10}-x_8x_9.$
 \end{tabular}
\end{center}
The other three can be described as determinants of the following matrices.
\begin{equation*}
 L=\begin{pmatrix}
  x_1&x_5&x_9&x_{13}\\
  x_2&x_6&x_{10}&x_{14}\\
  x_3&x_7&x_{11}&x_{15}\\
  x_4&x_8&x_{12}&x_{16}
 \end{pmatrix},\quad M=
\begin{pmatrix}
 x_1&x_9&x_3&x_{11}\\
 x_2&x_{10}&x_4&x_{12}\\
 x_5&x_{13}&x_7&x_{15}\\
 x_6&x_{14}&x_8&x_{16}
\end{pmatrix},\quad\tn{and}
\end{equation*}

\begin{equation*}
B=
\begin{pmatrix}
 x_1x_4-x_2x_3&x_1x_8+x_4x_5-x_3x_6-x_2x_7&x_5x_8-x_6x_7\\
 P_1(x_1,\dots,x_{16})&P_2(x_1,\dots,x_{16})&P_3(x_1,\dots,x_{16})\\
 x_9x_{12}-x_{10}x_{11}&x_9x_{16}+x_{12}x_{13}-x_{11}x_{14}-x_{10}x_{15}&x_{13}x_{16}-x_{14}x_{15}
\end{pmatrix}
\end{equation*}
$$ $$
\begin{tabular}{l}
where\\
$P_1(x_1,\dots,x_{16})=x_1x_{12}+x_4x_9-x_3x_{10}-x_2x_{11}$,\\
$P_2(x_1,\dots,x_{16})=x_1x_{16}+x_4x_{13}+x_5x_{12}+x_8x_9-x_3x_{14}-x_2x_{15}-x_7x_{10}-x_6x_{11}$,\\
$P_3(x_1,\dots,x_{16})=x_5x_{16}+x_8x_{13}-x_6x_{15}-x_{7}x_{14} $.
\end{tabular}

\begin{theorem}[\cite{MR2039690}]\label{thm:4qubitinvs}
 For the standard action $\SL(2,\C)^4\actson(\C^2)^{\ot 4}$, the invariant ring $$\C[(\C^2)^{\ot4}]^{\SL(2,\C)^4}=\C[H,\det(L),\det(M),\det(B)].$$
\end{theorem}

Now we consider the invariant ring of $\SL(2,\C)^4\actson((\C^2)^*)^{\ot 4}$ by the contragradient action. As it turns out, this action is isomorphic to the standard action of $\SL(2,\C)^4\actson(\C^2)^{\ot 4}$ by the self-duality of $\SL(2,\C)$. Given a matrix $M\in\SL(2,C)$, conjugating by $$T=\begin{pmatrix}0&-1\\1&0\end{pmatrix}$$ yields $(M^{-1})^T$. Let $\Theta=T^{\ot 4}$.

We define the linear map $\phi:(\C^2)^{\ot 4}\to((\C^2)^*)^{\ot 4}$ by $v\mapsto (\Theta v)^T$. Then for $g=(g_1,g_2,g_3,g_4)\in\SL(2,\C)^4$, $$\phi(g.v)=\phi((\ot_{i=1}^4{g_i})v)=(\Theta(\ot_{i=1}^4{g_i})\Theta^{-1} \Theta v)^T$$ 
$$=v^T\Theta^T(\ot_{i=1}^4{g_i^{-1}})=g.\phi(v).$$ Thus $\phi$ is an equivariant isomorphism. We define $H^\dagger:((\C^2)^*)^{\ot 4}\to \C$ by $H^\dagger=H(\phi(v)^T)$. We define $\det(L)^\dagger$, $\det(M)^\dagger$, and $\det(B)^\dagger$ likewise. 

\begin{corollary}\label{cor:dual4qubitinvs}
 For the contragradient action of $\SL(2,\C)^4\actson((\C^2)^*)^{\ot 4}$, the invariant ring $$\C[((\C^2)^*)^{\ot 4}]^{\SL(2,\C)^4}=\C[H^\dagger,\det(L)^\dagger,\det(M)^\dagger,\det(B)^\dagger].$$
\end{corollary}

Now note that there is an isomorphism of varieties $\Psi:\mcP_4\to(\mcP^\vee_4)^T$, where $(\mcP^\vee_4)^T\subseteq(\C^2)^{\ot 4}$ is the image of $\mcP^\vee_4$ under the transpose map. It is given by $(|0\rangle\langle1|+|1\rangle\langle0|)^{\ot 4}=\sum_{I\subseteq[4]}{|I\rangle\langle\overline{I}|}$. In fact, this map is an involution. Now we can express $\Theta$ as $(-|0\rangle\langle 1|+|1\rangle\langle0|)^{\ot 4}=\sum_{I\subseteq[4]}{(-1)^{|I|}|I\rangle\langle\overline{I}|}$. Since both $\mcP_4$ and $(\mcP^\vee_4)^T$ lie in the subspace $V$ of vectors $\sum_{I\subseteq[4]}{\alpha_I|I\rangle}$ where $\alpha_I=0$ if $|I|$ is odd, we see that $\Theta|_{V}=\Psi|_{V}\implies \Theta|_{\mcP_4}=\Psi|_{\mcP_4}$. So we have the following fact.

\begin{proposition}\label{prop:oneinvring}
 Let $\SL(2,\C)^4\actson \mcP_4$ by the standard action and $\SL(2,\C)^4\actson\mcP^\vee_4$ by the contragradient action, then $\C[\mcP_4]^{\SL(2,\C)^4}\cong\C[\mcP^\vee_4]^{\SL(2,\C)^4}.$
\end{proposition}
\begin{proof}
 By Theorem \ref{thm:4qubitinvs} and Corollary \ref{cor:dual4qubitinvs}, it suffices to show that $H(v)=H^\dagger(\Psi(v))$, $\det(L)(v)=\det(L)^\dagger(\Psi(v))$, etc. when $H,\det(L),\det(M),$ and $\det(B)$ are restricted to $\mcP_4$ and $H^\dagger,\det(L)^\dagger,\det(M)^\dagger$, and $\det(B)^\dagger$ are restricted to $\mcP_4^\vee$. 
 
 We see that for $v\in\mcP_4^\vee$, $H^\dagger(v)=H(\phi(v)^T)=H(\Theta v^T)=H(\Psi v^T)$ and $\Psi v^T\in\mcP_4$. A similar argument holds for the other three generators.
\end{proof}

As a consequence of this, to determine if the SWAP gate is in the orbit closure of $\SL(2,\C)^4\actson\mcP_4$ or the orbit closure of $\SL(2,\C)^4\actson\mcP^\vee_4$, we need only look at its values on the invariants $H,\det(L),\det(M),\det(B)$.

\subsection{The SWAP gate is not in the orbit closure of any Pfaffian (co)gate}

We first write down the SWAP gate as a vector and compute its values on the invariants $H,\det(L),\det(M)$, and $\det(B)$. As an element of $((\C^2)^*)^{\ot 2}\ot(\C^2)^{\ot 2}$, the SWAP gate is expressed as $|00\rangle\langle 00|+|01\rangle\langle10|+|10\rangle\langle01|+|11\rangle\langle11|$. Vectorizing it using the transpose map $(\C^2)^{\ot 2}\to((\C^2)^*)^{\ot 2}$, we get that the SWAP gate is the vector $|0000\rangle+|1010\rangle+|0101\rangle+|1111\rangle$. We can also vectorize it to an element of $((\C^2)^*)^{\ot 4}$ similarly, getting the vector $\langle 0000|+\langle 1010|+\langle 0101|+\langle 1111|$ If we index its coefficients by $\beta_1,\dots,\beta_{16}$, we have that $\beta_1=\beta_{6}=\beta_{11}=\beta_{16}=1$ and $\beta_{i}=0$ for all other $i$.

Plugging $\beta_i$ into $x_i$ into the generators of the invariant ring gives that $H(\tn{SWAP})=2$, $\det(L)(\tn{SWAP})=1$, $\det(M)(\tn{SWAP})=0$ and $\det(B)(\tn{SWAP})=0$. We let $I$ be the ideal formed by these polynomials. Then we note that vanishing locus of the ideal $J$ generated by the equations $\alpha_I=0$ for $|I|$ odd and the polynomial $$\alpha_{\emptyset}\alpha_{[4]}=\alpha_{\{1,2\}}\alpha_{\{3,4\}}-\alpha_{\{1,3\}}\alpha_{\{2,4\}}+\alpha_{\{2,3\}}\alpha_{\{1,4\}}$$ contains both $\mcP_4$ and $(\mcP^\vee_4)^T$ as subvarieties.

\begin{theorem}
 The \tn{SWAP} gate is not in the orbit closure of any Pfaffian (co)gate under the respective actions of $\SL(2,\C)^4$.
\end{theorem}
\begin{proof}
 By Proposition \ref{prop:oneinvring}, it suffices to show that $I+J=\C[x_1,\dots,x_{16}]$. Running this computation in Macaulay2 quickly yields that this is true.
\end{proof}

One may wonder if there is something to be gained by looking at the orbit closure of $\GL(2,\C)^4$ acting on the SWAP gate as it is a larger group. However, this is not the case. As previously mentioned, it was shown in \cite{margulies2013polynomial} that the SWAP gate is not in the $\GL(2,\C)^4$ orbit of any Pfaffian (co)gate. Thus we need only consider the boundary of these orbits. We recall a powerful tool in analyzing boundaries of orbits:

\begin{theorem}[The Hilbert-Mumford Criterion \cite{MR506989}]\label{thm:hbcriterion}
 For a linearly reductive group $G$ acting on a variety $V$, if $v\in\overline{G.w}\setminus G.w$ then there exists a 1-parameter subgroup (or cocharacter) $\lambda:k^\times\to G$ (where $\lambda$ is a homomorphism of algebraic groups), such that $\lim_{t\to 0}{\lambda(t).w}=v$. 
\end{theorem}

Now we note that $\lambda(t):=t(I_2)^{\ot 4}$ is a cocharacter of the action of $\GL(2,\C)^4$ on $(\C^2)^{\ot 4}$. We see that it sends every vector to the origin as $t\to0$. Thus every $\GL(2,\C)^4$ orbit closure intersects the origin an so the closures of any two orbits intersect. This makes looking at $\GL(2,\C)^4$ orbit closures rather meaningless.

\section{Tensors Orthogonal to Pfaffian Gates}\label{sec4}
 
 In this section, we find Pfaffian (co)gates that in a certain basis act identically to the SWAP gate, although other restrictions must be applied to the circuit. Thus certain cogates can be replaced with the SWAP gate and not change the value of the Pfaffian circuit. 
 
 We rely on Observation \ref{obs} and the fact that if a vector $v-u$ is in the orthogonal complement of $\mcP_n$ then $\langle u-v,P\rangle=0$ for any Pfaffian gate $P\in\mcP_n$. Thus $\langle u,P\rangle=\langle v,P\rangle$, leading to the following observation:
 \begin{observation}
  Suppose we have a Pfaffian circuit and a tensor $S$ equal to $P+Q$ where $P$ is a Pfaffian (co)gate in the circuit and $Q$ is orthogonal to every Pfaffian gate of the same size as in the circuit, then $P$ may be replace by $S$ without changing the value of the circuit.
 \end{observation}

 Under certain changes of basis, we show that the SWAP gate minus a Pfaffian cogate is in the orthogonal complement of the subspace containing the Pfaffian gates (after the change of basis).

 Suppose we have a Pfaffian circuit, which in reduced form has gate $\Xi=\sPf(X)$ and cogate $\Theta=\sPfd(T)$. We know that the value of the circuit is $\Theta(\Xi)=\sPf(X+\tilde{T})$ by Theorem \ref{thm:pfaffkernel}. The cogate $\Theta$ is the tensor product of several smaller Pfaffian cogates, some of which we want to replace with SWAP gates.

 \begin{defn}
We say that a vector $\sum_{I\subseteq[n]}{\beta_I|I\rangle}$ has \emph{odd support} if $\beta_I=0$ for $|I|$ even. We define even support similarly. We also use these terms when referring to covectors.  
 \end{defn}

 For this technique to work, we require that $T$, $X$ are $n\times n$, with $n$ even. We would like to perform a change of basis on the SWAP gate so that it is equal to $P+Q$, where $P$ is a Pfaffian cogate and $Q=\sum_{I\subseteq[4]}{\beta_I\langle I|}$ has odd support. That is, $Q$ lies in the orthogonal complement of the subspace containing $\mcP_4$. We can then switch one Pfaffian cogate with the SWAP tensor: let $\hat{\Theta}$ be the tensor product of all the Pfaffian cogates we are not changing. We have that $\hat{\Theta}=\sum_{I\subseteq[n']}{\gamma_I\langle I|}$ which has even support. Thus $Q\ot \hat{\Theta}$ lies in the orthogonal complement of the subspace containing the Pfaffian gate $\Xi$. 
 
 \begin{lemma}\label{lem:orbclosure}
  Viewing SWAP as a vector in either $((\C^2)^*)^{\ot 4}$ or $(\C^2)^{\ot 4}$, we have $\overline{\SL(2,\C)^4.\tn{SWAP}}\cong \SL(2,\C)^2$. The orbit $\overline{\GL(2,\C)^4.\tn{SWAP}}\cong\hat{\sigma}(\End(\C^2)^2)$, where $\hat{\sigma}$ denotes the affine cone over the Segre embedding. In particular, the orbit of $\tn{SWAP}$ under the group $\SL(2,\C)^4$ is closed while the $\GL(2,\C)^4$ orbit of $\tn{SWAP}$ is not closed.
 \end{lemma}
\begin{proof}
 It suffices to prove this for SWAP as a vector in $(\C^2)^{\ot 4}$. This is because both $\SL(2,\C)$ and $\GL(2,\C)$ are algebraic groups, meaning that inverse map is a morphism of varieties which will induce an isomorphism between the orbits in $((\C^2)^*)^{\ot 4}$ and $(\C^2)^{\ot 4}$. We recall that $$\tn{SWAP}=|0000\rangle+|0101\rangle+|1010\rangle+|1111\rangle$$ after applying the map $\phi:((\C^2)^*)^{\ot 2}\ot(\C^2)^{\ot 2}\to (\C^2)^{\ot 4}$ to the matrix $M_{\tn{SWAP}}=|00\rangle\langle00|+|01\rangle\langle10|+|10\rangle\langle01|+|11\rangle\langle11|$. The induced action by $\SL(2,\C)^4$ (and $\GL(2,\C)^4$) is given by $(M_4^T\ot M_3^T)M_{\tn{SWAP}}(M_1\ot M2)=M_{\tn{SWAP}}(M_3^TM_1\ot M_4^TM_1)$ since $M_{\tn{SWAP}}$ is the map sending $V_1\ot V_2\to V_2\ot V_1$, where $V_1,V_2\cong\C^2$. 
 
 But every element of $\SL(2,\C)$, $\GL(2,\C)$) can be written as $g^Th$ for $g,h\in\SL(2,\C)$, $g,h\in\GL(2,\C)$, respectively. As such, $\overline{\SL(2,\C)^4.\tn{SWAP}}$ is isomorphic to the algebraic closure $\sigma(\SL(2,\C)^2)\cong \SL(2,\C)^2$, where $\sigma$ is the Segre embedding. Similarly for $\overline{\GL(2,\C)^4.\tn{SWAP}}$, except the orbit closure is $\hat{\sigma}(\End(\C^2)^2)$, since it has a non-trivial algebraic closure and $\GL(2,\C)$ is a cone.
\end{proof}

 By Lemma \ref{lem:orbclosure}, it suffices to look at changes of basis of the form $M\ot N\ot I_2\ot I_2$ for $M,N\in\SL(2,\C)$. Let
 \begin{equation*}
  M=\begin{pmatrix}
     a&b\\
     c&d
    \end{pmatrix},\quad N=
    \begin{pmatrix}
     e&f\\
     g&h
    \end{pmatrix}
 \end{equation*}

\begin{tabular}{l l}
 $\tn{Then SWAP}(M\ot N\ot I_2\ot I_2)=$& $ae\langle0000|+bf\langle1100|+cg\langle0011|+$\\
 &$bg\langle1001|+cf\langle0110|+de\langle1010|+$\\
 &$ah\langle0101|+dh\langle1111|+Q$
 \end{tabular}
 
\noindent where $Q=\sum_{I\subseteq[4]}{\beta_I\langle I|}$ has odd support. We want the coefficients $ae$, $bf$, $cg$, $bg$, $cf$, $de$, $ah$, and $dh$ to satisfy the relations of being a Pfaffian cogate. That is, we want $dh=1$ and $ae=2bcfg-adeh$, which simplifies to $ae=bcfg$. This defines a variety of basis changes that let us replace a Pfaffian cogate with the SWAP gate. One solution is given by $d=h=1$, $a=e=\frac{1}{2}$, $b=f=\frac{1}{\sqrt{2}}$, and $c=g=-\frac{1}{\sqrt{2}}$. So we get our basis change from
\begin{equation*}
M=N= \begin{pmatrix}
 \frac{1}{2}&\frac{1}{\sqrt{2}}\\
 -\frac{1}{\sqrt{2}}&1
 \end{pmatrix}
\end{equation*}

\noindent The Pfaffian cogate is then given by $\sPfd(A)$ where
\begin{equation*}
 A=\begin{pmatrix}
    0&.5&.5&-.5\\
    -.5&0&-.5&.5\\
    -.5&.5&0&.5\\
    .5&-.5&-.5&0
   \end{pmatrix}.
\end{equation*}

So if we apply the basis change $\sPfd(A)(M^{-1}\ot M^{-1}\ot I_2\ot I_2)$, we can replace this Pfaffian cogate with the SWAP gate. More generally, for any solution of the above equations, we find a skew-symmetric matrix $S$, and then $\sPfd(S)(M^{-1}\ot N^{-1}\ot I_2\ot I_2)$ can be replaced with the SWAP gate. Note that this same construction allows a Pfaffian gate to be replaced with a SWAP gate. 

\begin{theorem}\label{thm:swapgates}
 There is a four dimensional variety of basis changes by $\SL(2,\C)^4$ such that SWAP$=P+Q$ where $P$ is Pfaffian and $Q$ has odd support.
\end{theorem}
\begin{proof}
 We note that the variety $\mcP_4^\vee$ has a polynomial parameterization by the $\sPfd$ map. The orbit of the SWAP tensor is parameterized by a choice of element in $\SL(2,\C)^2$. More precisely we have a the polynomial isomorphism $\Psi:\SL(2,\C)^2\to \SL(2,\C)^4.\tn{SWAP}$ given by $(M_1,M_2)\mapsto\sum{\alpha_I\langle I|}$ and  $\sPfd:\mathfrak{J}_4\to\mcP^\vee_4$ given by the $\sPfd(N)=\sum{\beta_i\langle I|}$. Then we find a Gr\"obner basis for the ideal $I:=\langle \alpha_I-\beta_I:\;|I|\tn{ even}\rangle+\langle \det(M_1)-1,\det(M_2)-1\rangle$ in Macaulay2.
\end{proof}

A caveat to this construction: we cannot replace more than one cogates we found in Theorem \ref{thm:swapgates} with a SWAP gate. This is because if the SWAP gate is written as $P+Q$ with $P$ a cogate and $Q$ with odd support under a change of basis, $(P+Q)^{\ot 2}=P\ot P+P\ot Q+Q\ot P+Q\ot Q$. We see that $P\ot P$ is again a cogate, $P\ot Q$, and $Q\ot P$ have odd support, but $Q\ot Q$ has even support. If $P\ot P+Q\ot Q$ is a cogate, only then this would be fine. 

In general, the more SWAP gates one wishes to replace, the more restrictions one gets. We need to look at $\tn{SWAP}^{\ot k}$ and try to write it as $P+Q$ as before.  In full generality, when we want to replace a Pfaffian cogate with another tensor, we are considering the following problem:

\begin{problem}
 Let $W_n\subset((\C^2)^*)^{\ot n}$ be the subspace of tensors of even support and $P_{W_n}:((\C^*)^2)^{\ot n}\to W$ be the associated linear projector. For a tensor $G\in((\C^2)^*)^{\ot n}$, determine if $P_{W_n}(\overline{\SL(2,\C)^n.G})\cap\mcP^\vee_n\ne\emptyset$ or if
 
 \noindent$P_{W_n}(\GL(2,\C)^n.G)\cap G\ne\emptyset$.
\end{problem}

 If the projection of the orbit closure of $G$ onto $W_n$ intersects the orbit closure of Pfaffian cogates, then it can be used as a cogate in a Pfaffian circuit, replacing a Pfaffian cogate after a suitable change of basis, without increasing the time complexity. We now look at the case of $\tn{SWAP}^{\ot k}$ for $k>1$ and we find that the above construction no longer works.

\begin{theorem}\label{thm:swap2}
 For $G=\tn{SWAP}^{\ot 2}$, $$\P_{W_8}(\overline{\SL(2,\C)^8.G})\cap\C\mcP^\vee_8=$$ $$P_{W_8}(\GL(2,\C)^8.G)\cap \C^* \mcP^\vee_8=\emptyset.$$
\end{theorem}
\begin{proof}
 
 We first looked at the ideal defining $\P_{W_8}(\overline{\SL(2,\C)^8.G})$ which is parameterized by a choice of four matrices in $\SL(2,\C)$. Intersecting this ideal with the ideal parameterizing $\C \mcP_8^\vee$, without the relation $a_{|1\cdots1\rangle}=1$, in Macaulay2 revealed that the intersection of the closures of these sets were empty. 
 
 We then considered the ideal $I=V(P_{W_8}(\End(\C^2)^8.G)\cap \mcP_8^\vee)$. We note that since $P_{W_8}(\GL(2,\C)^8.G)$ is closed under multiplication by elements of $\C^*$, it suffices to check that 
 $P_{W_8}(\GL(2,\C)^8.G)\cap \mcP^\vee_8=\emptyset.$ 

The space $P_{W_8}(\End(\C^2)^8.G)$ is parameterized by choosing $M_1,\dots,M_4\in\End(\C^2)$. We found that $I$ contained the polynomial $\det(M_1)\cdots\det(M_4)$ using Macaulay2. Thus at least one of the matrices parameterizing the intersection is not invertible. This shows that 
 $$P_{W_8}(\GL(2,\C)^8.G)\cap \mcP^\vee_8=\emptyset.$$
\end{proof}

As it turns out, the inability to replace a Pfaffian cogate with $\tn{SWAP}^{\ot 2}$ will make it impossible to replace a Pfaffian cogate with $\tn{SWAP}^{\ot k}$ for $k>2$. Let us write $\tn{SWAP}=P_{W_4}(\tn{SWAP})+P_{W_4^\perp}(\tn{SWAP})$ and then see that 
$$\tn{SWAP}^{\ot k}=\sum_{\stackrel{U_j=W_4,W_4^\perp}{j\in[k]}}{\ott_{i=1}^k{P_{U_i}(\tn{SWAP})}}.$$ Furthermore, each of these summands lie in pairwise orthogonal subspaces. Looking at $P_{W_{4k}}(\tn{SWAP}^{\ot k})$, this projects onto those subspaces $\ott_{i=1}^{k}{U_i}$ where an even number of $U_j=W_4^\perp$.

\begin{theorem}\label{thm:nomoreswaps}
 For $k\ge2$, we have that for $G=\tn{SWAP}^{\ot k}$, $$\P_{W_{4k}}(\overline{\SL(2,\C)^{4k}.G})\cap\C\mcP^\vee_{4k}=$$ $$P_{W_{4k}}(\GL(2,\C)^{4k}.G)\cap\C^* \mcP^\vee_8=\emptyset.$$
\end{theorem}
\begin{proof}
 We look at the relations defined in Theorem \ref{thm:rels}. We let $\overline{S}=\overline{T}=\emptyset$ and $\overline{R}=[8]$ (where $\overline{S}$ is the complement of $S$ since we are looking at cogates). This corresponds to looking at the summands of $$P_{W_4}(\tn{SWAP})\ot P_{W_4}(\tn{SWAP})\ot\cdots\ot P_{W_4}(\tn{SWAP})\quad\tn{and}$$
 $$P_{W^\perp_4}(\tn{SWAP})\ot P_{W^\perp_4}(\tn{SWAP})\ot P_{W_4}(\tn{SWAP})\ot\cdots\ot P_{W_4}(\tn{SWAP})$$ of $P_{W_{4k}}(\tn(SWAP)^{\ot k})$. Now let $\beta$ be the coefficient of the vector $|0\cdots 0\rangle$ in the vector $P_{W_{4(k-2)}}(\tn{SWAP}^{\ot(k-2)})$. The relations induced by this choice of subset are actually the Pfaffian cogate relations on $\beta P_{W_8}(\tn{SWAP}^{\ot 2})$ which we know from Theorem \ref{thm:swap2} has no solution.
 
\end{proof}

We note that the for any $G$ that one wishes to replace a Pfaffian cogate with, if $k$ copies cannot be placed into a Pfaffian circuit, then $k+1$ cannot either, using arguments completely analogous to those in Theorem \ref{thm:nomoreswaps}.

\section{Conclusion}

In this paper we have introduced two algebraic methods for altering Pfaffian circuits so that value remains unchanged. This were using orbit closures as well as the fact that the standard pairing of vectors is degenerate when restricted to Pfaffian (co)gates. 

The point of considering these algebraic approximations is to increase the gates available to those designing algorithms using tensor networks. We have considered the SWAP gate as a proof of concept as it is a very natural gate and an active area of research is focused on the relationship between planarity, computational complexity, and Pfaffian circuits.

This methods, as we have shown, allow for a single SWAP gate to be approximated by a Pfaffian circuit but it seems that planarity is very strongly tied to these networks. Still, for considering other gates, it may be that these methods prove to be more robust. 

\subsection*{Acknowledgments} The research leading to these results has received funding from the European Research Council under the European Union’s Seventh Framework Programme (FP7/2007-2013) / ERC grant agreement No 339109.

 \bibliographystyle{plain}
\bibliography{bibfile}

\end{document}